\documentclass[11pt,draft]{amsart}
\usepackage{dsfont}
\usepackage{amsfonts}
\usepackage{mathrsfs}
\usepackage{amssymb}
\usepackage{amsmath}
\usepackage{amsfonts}
\usepackage{amsfonts,enumerate}
\usepackage{pifont}
\usepackage{enumerate}
\usepackage{graphics}
\usepackage{verbatim}
\usepackage{amsthm}
\usepackage{latexsym,bm}
\usepackage{color}
\numberwithin{equation}{section}
\newtheorem{theorem}{Theorem}[section]

\newtheorem{lemma}[theorem]{Lemma}
\newtheorem{proposition}[theorem]{Proposition}

\theoremstyle{definition}

\newcommand{\R}{{\mathbb R}}

\begin{document}

\title[]{Banach lattice-valued $q$-variation and  convexity}

\author{Guixiang Hong}
\address{Instituto de Ciencias Matem\'aticas,
CSIC-UAM-UC3M-UCM, Consejo Superior de Investigaciones
Cient\'ificas, C/Nicol\'as Cabrera 13-15. \newline 28049, Madrid. Spain.\\
\emph{E-mail address: guixiang.hong@icmat.es}}
\thanks{The author is supported by MINECO: ICMAT Severo Ochoa project SEV-2011-0087 and ERC Grant StG-256997-CZOSQP (EU)}

\thanks{\small {{\it MR(2000) Subject Classification}.} Primary
42B25, 46B20; Secondary 46B99.}
\thanks{\small {\it Keywords.}
Variational inequalities, Hardy-Littlewood property, Behaviour in $L^\infty(\R)$.}

\maketitle

\begin{abstract}
In this paper, we show that the $q$-variation for differential operator is not bounded in $L^p(\mathbb{R};L^{\infty}(\mathbb{R}))$ for any $1<p<\infty$. As a consequence, the $q$-variation operator can not be used to characterize the Hardy-Littlewood property of the underlying Banach lattice. Moreover, for K\"othe function spaces $X$ with $X^*$ norming such that $X$ is $r$-convex for some large $r$, and $X$ is not $s$-convex for any $s$, $r<s<\infty$, we obtain lower bounds of the $(L^p(\mathbb{R};X),L^p(\mathbb{R};X)$-bounds of the $q$-variation operator, which tends to $\infty$, as $r$ tends to $\infty$. 
\end{abstract}

\section{Introduction}
In recent years, many research papers in probability, ergodic theory and harmonic analysis (see e.g. the references appearing in the Introduction of \cite{HoMa}) have been devoted to the study of the boundedness of the $q$-variation operators, with $2<q<\infty$, acting on scalar-valued functions. The $q$-variation operators can be viewed as `better' operators than the maximal operators in the sense that they immediately imply the pointwise convergence of the underlying family of operators without
using the Banach principle via the corresponding maximal inequality, and they can be used to measure the speed of convergence of the family. Very recently, vector-valued $q$-variations for differential operators and semigroups, i.e. $q$-variation operators acting on vector-valued functions, have also been considered  in \cite{HoMa} \cite{HoMa2}.

Let $X$ be a Banach lattice. Given a locally integrable function $f:\mathbb{R}\rightarrow X$, for any $t>0$, the differential averages $A_t$ is defined as
$$A_tf(x)=\frac{1}{t}\int^t_{-t}f(x-y)dy.$$
Let $J$ be a finite subset of $\mathbb{R}_+$, the Hardy-Littlewood operator $\mathcal{M}_J$ on $X$-valued functions $f$ is defined as
$$\mathcal{M}_Jf(x)=\sup_{t\in J}|A_tf(x)|,$$
where the sup is a sup in the lattice $X$. This accounts for the need to take a finite collection of radii $J$. When $X$ is a K\"othe function space, then we can take the sup on $\mathbb{R}_+$. 
We refer the readers to \cite{LiTz79} for more information on Banach lattices and Banach function spaces.
The $q$-variation operator $\mathcal{V}_{q,J}(A)$ on $X$-valued functions $f$ is defined as
$$\mathcal{V}_{q,J}(A)f(x)=\sup_{\{t_i\}\subset J}\big(\sum_i|A_{t_i}f(x)-A_{t_{i+1}}f(x)|^q\big)^{1/q}$$
where the supremum is taken over all decreasing sequences $\{t_i\}$ in $J$. It is obvious that the $q$-variation operator is bigger than the maximal operator, hence if $X$ is a Banach lattice such that $\mathcal{V}_{q,J}(A)$ is bounded in $L^p(\mathbb{R};X)$, then so does $\mathcal{M}_J$.
It is a natural question that whether the reverse implication could remain true?  

In this paper, we give a negative answer. Precisely, we will show the following result.

\begin{theorem}\label{main theorem}
Let $2<q<\infty$, then $\mathcal{V}_{q,J}(A)$ is not uniformly (w.r.t. $J$'s) bounded in $L^p(\mathbb{R}; L^{\infty}(\mathbb{R}))$ for any $1<p<\infty$.
\end{theorem}

This result provides a negative answer of the above question, since it is well-known that $\mathcal{M}_J$'s are uniformly bounded on $L^p(\mathbb{R};L^{\infty}(\mathbb{R}))$. The idea of the proof is to construct $L^\infty(\mathbb{R})$-valued function such that we can compute the lower bound, which is partially inspired by the arguments for the behavior of $q$-variation for the heat semigroup in $L^{\infty}(\mathbb{R})$ developed in \cite{BCT11} by Betancor {\it et al}.

The following result gives a quantitative description of Theorem \ref{main theorem} in the Banach function space level.
Given a K\"othe function space $X$. Then we can define the following quantity for fixed $1<p<\infty$ and $2<q<\infty$,
\begin{align}\label{CrA}
C_X(A)=\sup_{f\in L^p(\mathbb{R};X)}\sup_J\frac{\|\mathcal{V}_{q,J}(A)f\|_{L^p(X)}}{\|f\|_{{L^p(X)}}},
\end{align}
which may equals infinity if $\mathcal{V}_{q,J}$'s are not uniformly bounded on $L^p(\mathbb{R};X)$. However by the result in \cite{HoMa2} that when $X$ is a UMD lattice, $\mathcal{V}_{q,J}$'s are indeed uniformly bounded on $L^p(\mathbb{R};X)$,  $C_X(A)$ is finite for UMD lattice $X$.

The result is stated as follows.
\begin{theorem}\label{main theorem 2}
Let $1<p<\infty$ and $2<q<\infty$. Let $X$ be a K\"othe function space with $X^*$ norming such that $X$ is $r$-convex for some large $r$, and $X$ is not $s$-convex for any $s$, $r<s<\infty$.  Then there exist a constant $C>0$ such that $C_{X}(A)\geq Cr^{1/q}/c_r$, where $c_r$ is the constant of $r$-convexity of $X$.

In particular, we have $C_{L^r(\mathbb{R})}(A)\geq Cr^{1/q}$ which tends to infinity as $r$ tends to infinity.
\end{theorem}

Note that if $q=\infty$, the variation operator reduces to the maximal operator, and Theorem \ref{main theorem 2} gives no new information.  The formulation and the proof of this result is partially motivated by the work \cite{HMST95} by Harboure {\it et al}.

Recall that a Banach lattice $X$ is said to have the Hardy-Littlewood (H.L.) property if there exists some $1<p<\infty$ such that the operators $\mathcal{M}_J$ are uniformly bounded in $L^p(\mathbb{R}; X)$. See \cite{BFR12} and the references therein for more information on this property. Theorem \ref{main theorem} imlies that the H.L. property can not ensure the uniformly $L^p(X)$-boundedness of $q$-variation. While  the UMD property of Banach lattice $X$ is sufficient for the uniformly $L^p(X)$-boundedness of $q$-variation with $2<q<\infty$ from the result in \cite{HoMa} mentioned previously. It is interesting to know whether the UMD property of $X$ is also necessary for the uniformly $L^p(X)$-boundedness of $q$-variation for all $2<q<\infty$. Until the moment of writing the paper, we have no idea about this.

We will show Theorem \ref{main theorem} in Section 3. Theorem \ref{main theorem 2} will be proved in Section 4. One intermediate step will be shown in Section 2.

Throughout this paper, by $C$ we always denote a positive constant that may vary from line to line.

\section{Reduction}
In this section, we will reduce the statements in Theorem \ref{main theorem} and \ref{main theorem 2}  for  the $q$-variation associated to the differential averages to similar statements for the one associated to the heat semigroup  on the real line.

Let $\{e^{-t\Delta}\}_{t>0}$ be the heat semigroup on $\mathbb{R}$ and $\{H_t\}_{t>0}$  be the associated kernels. It is well-known that $H_t(x)=1/\sqrt{t}H(x/\sqrt{t}),\;\mathrm{with}\;H(x)=1/\sqrt{4\pi}e^{-|x|^2/4}.$ Let $H=\{H_t\}_{t>0}$. Denote the corresponding $q$-variation by $\mathcal{V}_{q,J}(H)$.

In this section, $X$ is assumed to be a K\"othe function space on a measure space $(\Omega,\nu)$. Thus any $X$-valued measurable function on $\mathbb{R}$ can be viewed as a  measurable function on $\mathbb{R}\times X$. Then similar to the definition of $C_X(A)$, for fixed $1<p<\infty$ and $2<q<\infty$, we can define $C_X(H)$.

The main result of this section is formulated as follows, which is a variant of Lemma 2.4 in \cite{CJRW03}.
\begin{lemma}\label{reduction}
Let $X$ be a K\"othe function space on a measure space $(\Omega,\nu)$ having Fatou property. Then
there exists a positive constant $C$ such that $C_X(H)\leq CC_X(A)$.
\end{lemma}

\begin{proof}
Set $h(|x|)=H(x)$, then $h(t)$ is trivially differentiable and
$$\int^{\infty}_0|h'(t)|tdt=C<\infty.$$
Note that
$$h(s)=-\int^{\infty}_sh'(t)dt=-\int^{\infty}_0\chi_{[s,\infty)}(t)h'(t)dt.$$
Hence we have
$$H(x)=-\int^{\infty}_0\chi_{[s,\infty)}(|x|)h'(t)dt.$$

Fix a $X$-valued function $f$ on $\mathbb{R}$, $A_tf$ can be viewed as function on $\mathbb{R}\times \Omega$. Let $A_t(x)=1/t\chi_{[-t,t]}(x)$, $A_tf$ can be rewritten as
$$A_tf(x,\omega)=A_t\ast f(x,\omega).$$
Then
\begin{align*}
H_s(x)=\frac{1}{\sqrt{s}}H\big(\frac{|x|}{\sqrt{s}}\big)&=-\int^{\infty}_0\frac{1}{t\sqrt{s}}\chi_{[0,t]}\big(\frac{|x|}{\sqrt {s}}\big)h'(t)tdt\\&=-\int^{\infty}_0A_{t\sqrt{s}}(x)h'(t)tdt.
\end{align*}
Consequently,
$$H_sf(x,\omega)=-\int^{\infty}_0A_{t\sqrt{s}}f(x,\omega)h'(t)tdt.$$

Fix a  finite subset $J\subset\mathbb{R}_+$, we have
\begin{align*}
\mathcal{V}_{q,J}(H)&f(x,\omega)=\sup_{\{s_j\}\subset J}\big(\sum_j|H_{s_j}f(x,\omega)-H_{s_{j+1}}f(x,\omega)|^q\big)^{1/q}\\
&\leq \sup_{\{s_j\}\subset J}\big(\sum_j|\int^{\infty}_0(A_{t\sqrt{s_j}}f(x,\omega)-A_{t\sqrt{s_{j+1}}}f(x,\omega))h'(t)tdt|^q\big)^{1/q}\\
&\leq \sup_{\{s_j\}\subset J}\int^{\infty}_0\big(\sum_j|A_{t\sqrt{s_j}}f(x,\omega)-A_{t\sqrt{s_{j+1}}}f(x,\omega)|^q\big)^{1/q}|h'(t)|tdt\\
&\leq \int^{\infty}_0\sup_{\{s_j\}\subset J}\big(\sum_j|A_{t\sqrt{s_j}}f(x,\omega)-A_{t\sqrt{s_{j+1}}}f(x,\omega)|^q\big)^{1/q}|h'(t)|tdt\\
&\leq \sup_J\mathcal{V}_{q,J}(A)f(x,\omega)\cdot\int^{\infty}_0|h'(t)|tdt\leq C\sup_J\mathcal{V}_{q,J}(A)f(x,\omega).
\end{align*}
Using this pointwise estimates, we obtain the desired result.
\end{proof}

By Lemma \ref{reduction}, in order to prove Theorem \ref{main theorem}, it suffices to prove
\begin{align}\label{C8H}
C_{L^{\infty}(\mathbb{R})}(H)\geq M;\;\forall M>0.
\end{align}
Indeed, suppose Theorem \ref{main theorem} were not true, i.e. $C_{L^\infty(\mathbb{R})}(A)<\infty$, then by Lemma \ref{reduction},  $C_{L^{\infty}(\mathbb{R})}(H)<\infty$ which contradict with (\ref{C8H}).

\section{Proof of Theorem \ref{main theorem}}
This section is devoted to the proof of (\ref{C8H}), hence finishes the proof of Theorem \ref{main theorem}.
Let $a>1$. Define the functions
$$G=\sum^{-1}_{k=-\infty}(-1)^{k+1}\chi_{[a^k,a^{k+1})}.$$
The following estimate, which is  the (2.2) in \cite{BCT11}, is the starting point of the proof.
\begin{lemma}\label{lem: key estimate}
We can find some $a>1$ such that there exist $C>0$ and $j_0\in\mathbb{N}$ satisfying
\begin{align}\label{key estimate}
|H_{a^{-2j}}G(0)-H_{a^{-2(j+1)}}G(0)|\geq C
\end{align}
for all $j\geq j_0$.
\end{lemma}

Now we are at a position to prove Theorem \ref{main theorem}.
\begin{proof}
We define a function on $\mathbb{R}\times\mathbb{R}$
$$\tilde{G}(x,y)=G(x-y)\chi_{[-1,1]}(y).$$
It is easy to check that $G\neq0$ only if $-1<x<2$, and whence
$$\|\tilde{G}\|_{L^p(L^{\infty})}=\|\sup_{y\in [-1,1]}\tilde{G}(\cdot,y)\|_{p}\leq 3^{1/p}.$$
As a consequence, for any $j_1\geq j_0$ which has appeared in Lemma \ref{lem: key estimate}, we have
$$C_{L^{\infty}(\mathbb{R})}(H)\geq 3^{-1/p}\|\big(\sum^{j_1}_{j=j_0}|H_{a^{-2j}}\tilde{G}-H_{a^{-2(j+1)}}\tilde{G}|^q\big)^{1/q}\|_{L^p(L^{\infty})}.$$
It is easy to verify that
\begin{align*}
H_{a^{-2j}}\tilde{G}(x,y)&=\int_{\mathbb{R}}H_{a^{-2j}}(x-z)\tilde{G}(z,y)dz\\
&=\int_{\mathbb{R}}H_{a^{-2j}}(x-z){G}(z-y)\chi_{[-1,1]}(y)dz\\
&=\int_{\mathbb{R}}H_{a^{-2j}}(x-y-z){G}(z)\chi_{[-1,1]}(y)dz\\
&=H_{a^{-2j}}{G}(x-y)\chi_{[-1,1]}(y).
\end{align*}
Note that whenever $x\in(0,1)$, the interval $(x-1,x+1)$ contain the interval $[0,1]$. Therefore $C_{L^{\infty}(\mathbb{R})}(H)$ is not less than
\begin{align*}
 &3^{-1/p}\Big(\int^1_0\sup_{y\in(-1,1)}\big(\sum^{j_1}_{j=j_0}|H_{a^{-2j}}{G}(x-y)-H_{a^{-2(j+1)}}{G}(x-y)|^q\big)^{p/q}dx\Big)^{1/p}\\
&=3^{-1/p}\Big(\int^1_0\sup_{y\in(x-1,x+1)}\big(\sum^{j_1}_{j=j_0}|H_{a^{-2j}}{G}(y)-H_{a^{-2(j+1)}}{G}(y)|^q\big)^{p/q}dx\Big)^{1/p}\\
&\geq 3^{-1/p}\Big(\int^1_0\sup_{y\in[0,1]}\big(\sum^{j_1}_{j=j_0}|H_{a^{-2j}}{G}(y)-H_{a^{-2(j+1)}}{G}(y)|^q\big)^{p/q}dx\Big)^{1/p}\\
\end{align*}
On the other hand, changing the variable, for every $z\in\mathbb{R}$,
\begin{align*}
&|H_{a^{-2j}}{G}(y)-H_{a^{-2(j+1)}}{G}(y)|\\
&=\frac{1}{\sqrt{4\pi}}\left|\frac{1}{a^{-j}}\int_{\mathbb{R}}e^{-|y-z|^2/4a^{-2j}}G(z)dz-\frac{1}{a^{-{(j+1)}}}\int_{\mathbb{R}}e^{-|y-z|^2/4a^{-2{(j+1)}}}G(z)dz\right|\\
&=\frac{1}{\sqrt{4\pi}}\left|(-1)^j\int_{\mathbb{R}}e^{-u^2/4}g(u+a^{-j}y)\chi_{[0,a^{-j})}(u+a^{-j}y)du\right.\\
&\hspace{4em}\left.- (-1)^{j+1}\int_{\mathbb{R}}e^{-u^2/4}g(u+a^{-(j+1)}y)\chi_{[0,a^{-(j+1)})}(u+a^{-(j+1)}y)du\right|.
\end{align*}
Observe that when $t$ tends to 0,
$$\int_{\mathbb{R}}e^{-u^2/4}g(u+t)\chi_{[0,B)}(u+t)du\rightarrow \int_{\mathbb{R}}e^{-u^2/4}g(u)\chi_{[0,B)}(u)du$$
uniformly in $B\in(0,\infty)$. Hence, for all $j_0\leq j\leq j_1$, $y\rightarrow 0$ implies
$$|H_{a^{-2j}}{G}(y)-H_{a^{-2(j+1)}}{G}(y)|\rightarrow |H_{a^{-2j}}{G}(0)-H_{a^{-2(j+1)}}{G}(0)|$$
Thus,  there exists $\delta>0$ such that for all $|y|<\delta a^{-(j_1+1)}$ and  all $j_0\leq j\leq j_1$, we have
\begin{align}\label{y control 0}
|H_{a^{-2j}}{G}(y)-H_{a^{-2(j+1)}}{G}(y)|\geq \frac{1}{2}|H_{a^{-2j}}{G}(0)-H_{a^{-2(j+1)}}{G}(0)|.
\end{align}
Finally, by Lemma \ref{key estimate}, we obtain
\begin{align*}
C_{L^{\infty}(\mathbb{R})}(H) &\geq 3^{-1/p}\sup_{y\in[0,1]}\big(\sum^{j_1}_{j=j_0}|H_{a^{-2j}}{G}(y)-H_{a^{-2(j+1)}}{G}(y)|^q\big)^{1/q}\\
&\geq 3^{-1/p}\frac{C}{2}(j_1-j_0)\geq M,
\end{align*}
for any $M>0$, provided that $j_1$ is taken big enough. This finishes the proof of (\ref{C8H}), hence of Theorem \ref{main theorem}.
\end{proof}

\section{Proof of Theorem \ref{main theorem 2}}
The starting point of the proof is the following proposition, which is Theorem \ref{main theorem 2} in the case $X=L^r(\mathbb{R})$.
\begin{proposition}\label{pro:CrA}
Let $1<p<\infty$ and $2<q<\infty$. Then there exist a constant $C>0$ such that $C_{L^r(\mathbb{R})}(A)\geq Cr^{1/q}$.
\end{proposition}

The proof is postponed to the next section.
The proof of Theorem \ref{main theorem 2} is divided into two steps. The first step is to deduce $C_{\ell^r}(A)\geq Cr^{1/q}$ from $C_{L^r(\mathbb{R})}(A)\geq Cr^{1/q}$ which is Proposition \ref{pro:CrA}; The second step is to show $C_{X}(A)\geq Cr^{1/q}/c_r$ from $C_{\ell^r}(A)\geq Cr^{1/q}$ using Proposition 3.11 of \cite{HMST95}. We state this proposition here for convenience.

\begin{lemma}\label{lem:finite representation of X}
Let $X$ be a K\"othe function space with $X^*$ norming such that $X$ is $r$-convex for some $r$, $1<r<\infty$ and $X$ is not $s$-convex for any $s$, $r<s<\infty$. Then given $\varepsilon$ and a positive integer $m$, there exists a sequence $\{e_i\}^m_{i=1}$ of pairwise disjoint elements of $X$ such that
\begin{align}\label{finite representation of X}
(1-\varepsilon)\sum^m_{i=1}b_i^r\leq\left\|\sum^m_{i=1}b_ie_i\right\|^r_X\leq c^r_r\sum^m_{i=1}b^r_i
\end{align}
holds for any sequence $\{b_i\}^m_{i=1}$ of non negative scalars and $c_r$ is the constant of $r$-convexity of $X$.
\end{lemma}

Now let us give the proof of Theorem \ref{main theorem 2}.
\begin{proof}
(i) $C_{\ell^r}(A)\geq Cr^{1/q}$. Since the set
$$S=\{f=\sum^m_{k=1}\big(\sum^n_{j=1}a_{jk}\chi_{F_j}(x)\big)\chi_{E_k}(y):\;a_{jk}\in\mathbb{C},\;F_j\subset\mathbb{R}\; ,E_k\subset\mathbb{R}\;\mathrm{disjoint}\}$$
is dense in $L^p(\mathbb{R};L^r(\mathbb{R}))$, by the definition of $C_{L^r(\mathbb{R})}(A)$, we can find some $J$ and $f\in S$ such that
$$\|\mathcal{V}_{q,J}(A)f\|_{L^p(L^r)}\geq Cr^{1/q}\|f\|_{L^p(L^r)}.$$
For $f$ in the form in $S$,
we define a $\ell^r$-valued  function $\tilde{f}$ on $\mathbb{R}$ as
$$\tilde{f}(x)=\{|E_k|^{1/r}\sum^n_{j=1}a_{jk}\chi_{F_j}(x)\}^m_{k=1}.$$
Then the desired estimate follows from the two identities
\begin{align*}
\|\tilde{f}\|_{L^p(\ell^r)}=\|f\|_{L^p(L^r)},
\end{align*}
and
\begin{align*}
\|\mathcal{V}_{q,J}(A)\tilde{f}\|_{L^p(\ell^r)}=\|\mathcal{V}_{q,J}(A)f\|_{L^p(L^r)}.
\end{align*}
Let us just explain the latter equality, since the first one is proved in a similar way. Denote $\sum^n_{j=1}a_{jk}\chi_{F_j}(x)$ by $F_{jk}(x)$. Fix $x\in\mathbb{R}$,  by the fact that $E_k$'s are disjoint,
\begin{align*}
\|\mathcal{V}_{q,J}(A)f(x,\cdot)\|^r_{L^r}&=\int_{\mathbb{R}}\sup_{\{t_i\}\subset J}\big(\sum_i|A_{t_i}f(x,y)-A_{t_{i+1}}f(x,y)|^q\big)^{\frac rq}dy\\
&=\int_{\mathbb{R}}\sup_{\{t_i\}\subset J}\big(\sum_i|\sum^m_{k=1}(A_{t_i}-A_{t_{i+1}})F_{jk}(x)\chi_{E_k}(y)|^q\big)^{\frac rq}dy\\
&=\sum^m_{k=1}\int_{\mathbb{R}}\sup_{\{t_i\}\subset J}\big(\sum_i|(A_{t_i}-A_{t_{i+1}})F_{jk}(x)\chi_{E_k}(y)|^q\big)^{\frac rq}dy\\
&=\sum^m_{k=1}|E_k|\sup_{\{t_i\}\subset J}\big(\sum_i|(A_{t_i}-A_{t_{i+1}})F_{jk}(x)|^q\big)^{\frac rq}\\
&=\left\|\Big\{\sup_{\{t_i\}\subset J}\big(\sum_i||E_k|^{\frac1r}(A_{t_i}-A_{t_{i+1}})F_{jk}(x)|^q\big)^{\frac 1q}\Big\}^m_{k=1}\right\|^r_{\ell^r}\\
&=\|\mathcal{V}_{q,J}(A)\tilde{f}\|^r_{\ell^r}.
\end{align*}

(ii) $C_{X}(A)\geq Cr^{1/q}/c_r$. Similarly, by density argument, we can find some $J$ and $f\in L^p(\mathbb{R};\ell^r)$ of the form
$$f(x)=\big\{f_{k}(x)\big\}^{m}_{k=1},\;\mathrm{with}\; f_{k}\in L^p(\mathbb{R})$$
such that
$$\|\mathcal{V}_{q,J}(A)f\|_{L^p(\ell^r)}\geq Cr^{1/q}\|f\|_{L^p(\ell^r)}.$$
Fix $x\in\mathbb{R}$. Take $\varepsilon=1/2$, by Lemma \ref{lem:finite representation of X},  there exists a sequence $\{e_k\}^m_{k=1}$ of pairwise disjoint elements of $X$ such that
\begin{align*}
\left\|\sum^m_{k=1}f_{k}(x)e_k\right\|^r_X\leq c^r_r\sum^m_{k=1}f^r_k(x)=c^r_r\|f(x)\|^r_{\ell^r}
\end{align*}
and
\begin{align*}
\frac{1}{2}\|\mathcal{V}_{q,J}(A)f(x)\|^r_{\ell^r}=\frac{1}{2}\sum^m_{k=1}\big(\mathcal{V}_{q,J}(A)f_k(x)\big)^r\leq\left\|\sum^m_{k=1}\big(\mathcal{V}_{q,J}(A)f_k(x)\big)e_k\right\|^r_X.
\end{align*}
Now define $X$-valued function $\tilde{f}(x)=\sum^m_{k=1}f_{k}(x)e_k$, using the disjoint property of $e_k$'s,  it is easy to check that
$$\|\mathcal{V}_{q,J}(A)\tilde{f}(x)\|_X=\left\|\sum^m_{k=1}\big(\mathcal{V}_{q,J}(A)f_k(x)\big)e_k\right\|_X.$$
To conclude, using the result obtained in first step, we deduce that
\begin{align*}
\|\mathcal{V}_{q,J}(A)\tilde{f}\|_{L^p(X)}&\geq 2^{-1/r}\|\mathcal{V}_{q,J}(A)f\|_{L^p(\ell^r)}\\
&\geq 2^{-1/r}Cr^{1/q}\|f\|_{L^p(\ell^r)}\geq Cr^{1/q}/c_r\|\tilde{f}\|_{L^p(X)},
\end{align*}
which implies the desired estimate.
\end{proof}

\section{Proof of Proposition \ref{pro:CrA}}
As in section 2, by Lemma \ref{reduction}, to prove Proposition \ref{pro:CrA}, it suffices to prove
\begin{align}\label{CrH}
C_{L^r{\mathbb{R}}}(H)\geq Cr^{1/q}.
\end{align}
Now we adapt the previous argument for the proof of (\ref{C8H}) to the proof of (\ref{CrH}). Let us explain it.
\begin{proof}
Take the function $\tilde{G}$ as in the proof of Theorem \ref{main theorem}. It is easy to check that
$$\|\tilde{G}\|_{L^p(L^r)}\leq 2^{1/r}3^{1/p}.$$
As a consequence, for any $j_1\geq j_0$ which has appeared in Lemma \ref{lem: key estimate}, we have
$$C_{L^{r}(\mathbb{R})}(H)\geq 2^{-1/r}3^{-1/p}\|\big(\sum^{j_1}_{j=j_0}|H_{a^{-2j}}\tilde{G}-H_{a^{-2(j+1)}}\tilde{G}|^q\big)^{1/q}\|_{L^p(L^{r})}.$$
As in the proof of Theorem \ref{main theorem}, we arrive at the step that $C_{L^{r}(\mathbb{R})}(H)$ is not less than
$$2^{-1/r}3^{-1/p}\left(\int^1_0\Big(\int^1_0\big(\sum^{j_1}_{j=j_0}|H_{a^{-2j}}{G}(y)-H_{a^{-2(j+1)}}{G}(y)|^q\big)^{r/q}dy\Big)^{p/r}dx\right)^{1/p}.$$
Now using the estimate (\ref{y control 0}) for all $|y|<\delta a^{-(j_1+1)}$ and  all $j_0\leq j\leq j_1$, we get
\begin{align*}
&C_{L^{r}(\mathbb{R})}(H)\geq 2^{-1/r}3^{-1/p}\Big(\int^{1}_0\big(\sum^{j_1}_{j=j_0}|H_{a^{-2j}}{G}(y)-H_{a^{-2(j+1)}}{G}(y)|^q\big)^{r/q}dy\Big)^{1/r}\\
&\geq 2^{-1/r}3^{-1/p}\Big(\int^{\delta a^{-(j_1+1)}}_0\big(\sum^{j_1}_{j=j_0}|H_{a^{-2j}}{G}(y)-H_{a^{-2(j+1)}}{G}(y)|^q\big)^{r/q}dy\Big)^{1/r}\\
&\geq 2^{-1/r}3^{-1/p}\frac{C}{2}(j_1-j_0)^{\frac{1}{q}}({\delta a^{-(j_1+1)}})^{{1/r}}=\frac{C}{2}3^{-1/p}(\frac{\delta}{2a})^{1/r}a^{-j_1/r}(j_1-j_0)^{1/q}.
\end{align*}
Taking $j_1=[r]j_0$, note that $({\delta}/{2a})^{1/r}\rightarrow 1$ and $a^{-j_1/r}\rightarrow a^{-j_0}$ as $r\rightarrow\infty$. Therefore, we obtain the desired result, i.e. for large $r$, we have
\begin{align*}
C_{L^{r}(\mathbb{R})}(H)\geq Cr^{1/q}.
\end{align*}
\end{proof}

Proposition \ref{pro:CrA} should be compared to the following similar behaviour of Hilbert transform. 
\begin{proposition}\label{pro:lower bound of H}
For any $1<p<\infty$, there exist a positive constant $C$ such that
\begin{align}\label{lower bound of H}
\|H\|_{L^p(\mathbb{R};L^r(\mathbb{R}))\rightarrow L^p(\mathbb{R};L^r(\mathbb{R}))}\geq Cr
\end{align}
for sufficiently large $r$.
\end{proposition}

On the other hand, it is trivial that the Hardy-Littlewood maximal function satisfies
\begin{align*}\label{lower bound of M}
\|M\|_{L^p(\mathbb{R};L^r(\mathbb{R}))\rightarrow L^p(\mathbb{R};L^r(\mathbb{R}))}\geq C.
\end{align*}
Hence the $q$ variation operator can be regarded as an operator between singular integral operators and maximal operator. 

Proposition \ref{pro:lower bound of H} should have been known somewhere, but we do not find it in any literature. Hence we give a proof here. 

\begin{proof}
We will consider the function on $\mathbb{R}\times\mathbb{R}$
$$F(x,y)=f(x-y)\chi_{[-1,1]}(y)$$
with $f=\chi_{[0,1)}$ defined on $\mathbb{R}$.
Then it is easy to check that $F\neq0$ only if $-1<x<2$, and
$$\|F\|_{L^p(L^r)}\leq 2^{1/r}3^{1/p}.$$
Hence
$$\|H\|_{L^p(L^r)\rightarrow L^p(L^r) }\geq 2^{-1/r}3^{-1/p}\|H(F)\|_{ L^p(L^r)}.$$
On the other hand, it is easy to check that
$$HF(x,y)=Hf(x-y)\chi_{[-1,1]}(y).$$
Therefore using the fact that $(0,1)\subset (x-1,x+1)$ whenever $x\in (0,1)$, we get
\begin{align*}
\|HF\|_{ L^p(L^r)}&\geq \left(\int^1_0\Big(\int^1_{-1}|Hf(x-y)|^rdy\Big)^{p/r}dx\right)^{1/p}\\
&=\left(\int^1_0\Big(\int^{x+1}_{x-1}|Hf(y)|^rdy\Big)^{p/r}dx\right)^{1/p}\\
&\geq \left(\int^1_0\Big(\int^{1}_{0}|Hf(y)|^rdy\Big)^{p/r}dx\right)^{1/p}\\
&=\|Hf\|_{L^r([0,1])}.
\end{align*}
Using the fact that Hilbert transform is a principle value, that is
$$Hf(x)=\lim_{\varepsilon\rightarrow 0}\int_{|x-y|>\varepsilon}\frac{1}{x-y}f(y)dy.$$
It is easy to conclude by the cancelation condition of the kernel $p.v.1/x$ that for any $0<x<1/2$,
$$Hf(x)=\int^1_{2x}\frac{1}{x-y}dy=\ln\frac{x}{1-x}.$$
Consequently,  for large $M>0$ and $0<x\leq e^{-M}$,
$$|Hf(x)|=\ln\frac{1-x}{x}\geq \frac{1}{2}\ln\frac{1}{x}\geq \frac{M}{2},$$
whence
\begin{align*}
\|Hf\|_{L^r([0,1])}&\geq \|Hf\|_{L^r([0,e^{-M}])}\geq \frac{M}{2}e^{-M/r}.
\end{align*}
Taking $M=r$, we conclude that
$$\|H\|_{L^p(L^r)\rightarrow L^p(L^r) }\geq 2^{-1/r}3^{-1/p}(2e)^{-1}r$$
which implies the desired estimate (\ref{lower bound of H}).
\end{proof}
{\bf Acknowledgements}. The author would like to express his great appreciation to Professor J. L. Torrea for the discussions related to the topics in this paper. Especially, Theorem \ref{main theorem 2} is suggested by him to the author.

\vskip30pt


\begin{thebibliography}{0}

\bibitem{BCT11} J. J. Betancor, R. Crescimbeni and J. L. Torrea, The $\rho$-variation of the heat semigroup in the hermite setting: Behaviour in $L^{\infty}$, Proc. Edinburgh Math. Soc. 54 (2011), 1-17.

\bibitem{BFR12} J. J. Betancor, J. C. Fari\~na, L. Rodrguez-Mesa, Hardy-Littlewood and UMD Banach lattices via Bessel convolution operators,  J. Operator Theory.,  67 (2012), no. 2, 349-368.


\bibitem{CJRW03} J.T. Campbell, R.L. Jones, K. Reinhold, and M. Wierdl. Oscillation and variation for singular integrals in higher dimensions,  Trans. Amer. Math. Soc., 355: 2115-2137, 2003.



\bibitem{HMST95} E. Harboure, R. A. Mac\'ias, C. Segovia, J. L. Torrea, Some estimates for maximal functions on K\"othe function spaces, Israel J. Math. 90 (1995), 349-371.

\bibitem{HoMa} G. Hong, T. Ma, Vector-valued $q$-variation for differential operators and semigroups I, Preprint.

\bibitem{HoMa2} G. Hong, T. Ma, Vector-valued $q$-variation for differential operators and analytic semigroups II, Preprint.

\bibitem{LiTz79} J. Lindenstrauss, L. Tzafriri, Classical Banach space II, Springer-Verlag Berlin Heidelberg New York, 1979.

\bibitem{Rub86} J.L. Rubio de Francia, Martingale and integral transforms of Banach space valued functions, Springer Lect. Notes in Math. 1221 (1986), 195-222.

\end{thebibliography}
\end{document}